\documentclass[12pt]{amsart}
\usepackage{amssymb}
\usepackage{amsmath, amscd}
\usepackage{amsthm}
\usepackage{amsmath}
\usepackage{comment}
\usepackage[table,xcdraw]{xcolor}
\usepackage{ulem}
\usepackage{tikz}
\usepackage{float}
\usepackage{graphicx}
\usepackage{circuitikz}
\usetikzlibrary{positioning}
\usepackage[colorlinks=true, linkcolor=blue,urlcolor=blue]{hyperref}

\newtheorem{theorem}{Theorem}[section]

\newtheorem{proposition}[theorem]{Proposition}
\newtheorem{lemma}[theorem]{Lemma}

\newtheorem{corollary}[theorem]{Corollary}
\theoremstyle{definition}

\newtheorem{example}[theorem]{Example}
\newtheorem{definition}[theorem]{Definition}

\begin{document}

\author[Sh. Najafi]{Shahram Najafi}
\address{Department of Mathematics, Tarbiat Modares University, 14115-111 Tehran Jalal AleAhmad Nasr, Iran}
\author[A. Moussavi]{Ahmad Moussavi$^*$}
\address{Department of Mathematics, Tarbiat Modares University, 14115-111 Tehran Jalal AleAhmad Nasr, Iran}
\email{moussavi.a@modares.ac.ir; moussavi.a@gmail.com}
\author[P. Danchev]{Peter Danchev}
\address{Institute of Mathematics and Informatics, Bulgarian Academy of Sciences, 1113 Sofia, Bulgaria}
\email{danchev@math.bas.bg; pvdanchev@yahoo.com}

\thanks{$^*$Corresponding author: Ahmad Moussavi, email: moussavi.a@modares.ac.ir; moussavi.a@gmail.com}

\title[2-UQ rings]{Rings Whose Units Have Identity Plus Quasi-Nilpotent Square}

\begin{abstract}
In this paper, we investigate the structural and characterizing properties of the so-called {\it 2-UQ rings}, that are rings such that the square of every unit is the sum of an idempotent and a quasi-nilpotent element that commute with each other. We establish some fundamental connections between 2-UQ rings and relevant widely classes of rings including 2-UJ, 2-UU and tripotent rings. Our novel results include: (1) complete characterizations of 2-UQ group rings, showing that they force underlying groups to be either 2-groups or 3-groups when $3 \in J(R)$; (2) Morita context extensions preserving the 2-UQ property when trace ideals are nilpotent; and (3) the discovery that potent 2-UQ rings are precisely the semi-tripotent rings. Furthermore, we determine how the 2-UQ property interacts with the regularity, cleanness and potent conditions. Likewise, certain examples and counter-examples illuminate the boundaries between 2-UQ rings and their special relatives.

These achievements of ours somewhat substantially expand those obtained by Cui-Yin in Commun. Algebra (2020) and by Danchev {\it et al.} in J. Algebra \& Appl. (2025).
\end{abstract}

\maketitle

\section{Introduction}

In the present article, let $R$ be an associative ring with unity that is not necessarily commutative. For such a ring, we employ the following standard notations: we stand $U(R)$ for the set of invertible elements, $Nil(R)$ for the set of nilpotent elements, $C(R)$ for the set of central elements, $Id(R)$ for the set of idempotent elements, $QN(R)$ for the set of quasi-nilpotent elements, and $J(R)$ for the Jacobson radical. Also, the ring of $n \times n$ matrices over $R$ is denoted by $M_n(R)$, while the symbol $T_n(R)$ represents the ring of $n \times n$ upper triangular matrices over $R$. Following the classical terminology, a ring is called {\it abelian} if all its idempotents are central, that is, $Id(R) \subseteq C(R)$.

As usual, an element $a \in R$ is said to be {\it quasi-nilpotent} if $1-ax$ is invertible for each $x \in R$ satisfying $xa = ax$. Consequently, both $Nil(R)$ and $J(R)$ are obviously subsets of $QN(R)$. It should be emphasized that quasi-nilpotent elements play a critical role in analyzing the structure of Banach algebras $\mathcal{A}$. These elements have led to the introduction of several important concepts such as (strongly) J-clean rings \cite{chen2010}, nil-clean rings \cite{diesl2013}, generalized Drazin inverses \cite{koliha1996}, and quasi-polar rings \cite{chen2012}.

It is well known that $1+J(R) \subseteq U(R)$ for any ring $R$. Following \cite{D} or \cite{14}, $R$ is termed a {\it UJ-ring} whenever the equality $U(R)=1+J(R)$ holds. More generally, building upon \cite{cui}, in which work the idea mainly imitates that of \cite{D1}, we say that $R$ is a {\it 2-UJ ring} if, for every unit $u\in U(R)$, there exists $j\in J(R)$ such that $u^2=1+j$. This class of rings properly generalizes the class of UJ-rings. As demonstrated in \cite{cui}, for 2-UJ rings, the properties of being semi-regular, exchange and clean all do coincide.

In a related direction, mimicking \cite{12}, $R$ is named a {\it UU-ring} when $U(R)=1+Nil(R)$. Extending this notion, Sheibani and Chen \cite{10} defined {\it 2-UU rings} as those rings for which the square of each unit equals $1_R$ plus a nilpotent element. Their work confirmed that $R$ is strongly 2-nil-clean exactly when $R$ is an exchange 2-UU ring.

Let us now recall some fundamental concepts essential for our discussion: A ring \( R \) is called {\it Boolean} if all its elements are idempotents. In general, a ring \( R \) is called {\it tripotent} if every element \( a \in R \) satisfies the equality \( a^3 = a \); such elements are referred to as {\it tripotent elements}. In a more general setting, a ring \( R \) is said to be {\it regular} (resp., {\it unit-regular}) in the von-Neumann sense if, for each \( a \in R \), there exists \( x \in R \) (resp., \( x \in U(R) \)) such that \( a=axa \). Further, \( R \) is said to be {\it strongly regular} if, for all \( a \in R \), \( a = a^2b \) for some $b\in R$ depending on $a$. \\ In an other vein, a ring \( R \) is termed {\it exchange} if, for every \( a \in R \), there is an idempotent \( e \in aR \) with \( 1-e \in (1-a)R \). Besides, \( R \) is named {\it clean} if each its element can be expressed as the sum of an idempotent and a unit (see, for more details, \cite{nicc}). While every clean ring is exchange, the converse does {\it not} generally hold, though it is true for abelian rings (cf. \cite[Proposition 1.8]{nicc}). Additionally, \( R \) is called {\it semi-regular} if the quotient \( R/J(R) \) is regular and all idempotents lift modulo \( J(R) \). Semi-regular rings are known to be exchange, but the converse is however {\it not} always valid (see \cite{nicc}).

Next, Chen \cite{chen2010} introduced the class of strongly J-clean rings; in fact, an element is {\it strongly J-clean} if it is the sum of an idempotent and an element from the Jacobson radical that commute with each other, and a ring is {\it strongly J-clean} if all its elements are strongly J-clean or, equivalently, if $R$ is strongly clean and $R/J(R)$ is Boolean. In addition, it was shown in \cite[Theorem 4.1.9]{chenbook} that $R$ is a strongly J-clean ring if, and only if, $R$ is strongly clean and UJ.

Later, in 2025, Danchev {\it et al.} introduced the so-termed {\it UQ rings} and {\it strongly quasi-nil clean rings}, which are non-trivial extensions of {\it UJ rings} and {\it strongly J-clean rings}, respectively (see \cite{daoa}). In fact, a ring is called {\it UQ} if $U(R) = 1+ QN(R)$, and a ring $R$ is called {\it strongly quasi-nil clean} if every element of $R$ is the sum of an idempotent and a quasi-nilpotent element that commute with each other. They proved that $R$ is a {\it strongly quasi-nil clean ring} if, and only if, $R$ is strongly clean and UQ.

As a logical and non-trivial generalization of these concepts, we define the class of {\it 2-UQ rings} as follows: a ring $R$ is {\it 2-UQ} if, for every unit $u\in U(R)$, its square $u^2$ equals the sum of an idempotent and a quasinilpotent element that commute with each other (or, equivalently, $u^2=1+q$, where $q\in QN(R)$). Evidently, this class properly contains all {\it UQ} rings (and, consequently, all unit uniquely clean rings from \cite{13}) and rings with exactly two units. While all {\it 2-UJ} rings (and thus, particularly, all {\it UJ} rings) are {\it 2-UQ}, the converse is manifestly {\it not} true as we will demonstrate in the sequel.

Our goal here is to characterize these 2-UQ rings by analyzing their crucial properties in a close relation to {\it 2-UU} and {\it 2-UJ} rings, respectively, and to discover novel properties of the 2-UQ rings that are uncommon in current studies. Concretely, we succeeded to prove that, for potent rings, the 2-UQ property amounts to the 2-UJ property, and for semi-potent rings this holds modulo the Jacobson radical (see, for a more information, Theorems~\ref{semipotent} and \ref{exe 2}, respectively). About the inheritance of the 2-UQ property pertained to group rings, we successfully establish that any group ring equipped with that property obligates the full group to be either a 2-group or a 3-group, respectively, provided that either 2 lies in the Jacobson radical, or 3 lies in the Jacobson radical and the whole group is $p$-primary (see Proposition~\ref{group ring 2 in Jacobson} and Theorem~\ref{23}, respectively).

\medskip
\medskip

In order to illustrate all of what we said above, the next diagram naturally sheds some light on the transversal between the already discussed sorts of rings.

\medskip

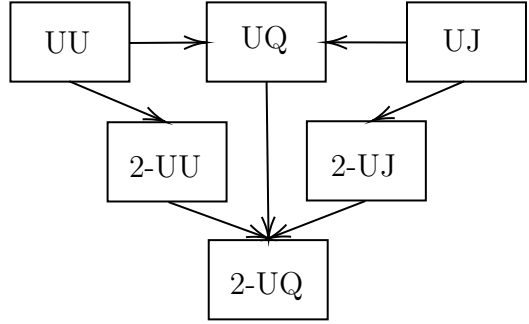
\begin{center}

\tikzset{every picture/.style={line width=0.75pt}} 

\begin{tikzpicture}[x=0.75pt,y=0.75pt,yscale=-1,xscale=1]

\draw   (41,50) -- (101,50) -- (101,90) -- (41,90) -- cycle ;
\draw    (101.17,70.5) -- (137.5,70.26) ;
\draw [shift={(139.5,70.25)}, rotate = 179.63] [color={rgb, 255:red, 0; green, 0; blue, 0 }  ][line width=0.75]    (10.93,-3.29) .. controls (6.95,-1.4) and (3.31,-0.3) .. (0,0) .. controls (3.31,0.3) and (6.95,1.4) .. (10.93,3.29)   ;
\draw    (120.73,150.8) -- (169.29,168.48) ;
\draw [shift={(171.17,169.17)}, rotate = 200.01] [color={rgb, 255:red, 0; green, 0; blue, 0 }  ][line width=0.75]    (10.93,-3.29) .. controls (6.95,-1.4) and (3.31,-0.3) .. (0,0) .. controls (3.31,0.3) and (6.95,1.4) .. (10.93,3.29)   ;
\draw    (240.5,70.25) -- (201.83,70.17) ;
\draw [shift={(199.83,70.17)}, rotate = 0.12] [color={rgb, 255:red, 0; green, 0; blue, 0 }  ][line width=0.75]    (10.93,-3.29) .. controls (6.95,-1.4) and (3.31,-0.3) .. (0,0) .. controls (3.31,0.3) and (6.95,1.4) .. (10.93,3.29)   ;
\draw    (270,89.75) -- (225.34,108.96) ;
\draw [shift={(223.5,109.75)}, rotate = 336.73] [color={rgb, 255:red, 0; green, 0; blue, 0 }  ][line width=0.75]    (10.93,-3.29) .. controls (6.95,-1.4) and (3.31,-0.3) .. (0,0) .. controls (3.31,0.3) and (6.95,1.4) .. (10.93,3.29)   ;
\draw    (220,150.25) -- (173.03,168.44) ;
\draw [shift={(171.17,169.17)}, rotate = 338.83] [color={rgb, 255:red, 0; green, 0; blue, 0 }  ][line width=0.75]    (10.93,-3.29) .. controls (6.95,-1.4) and (3.31,-0.3) .. (0,0) .. controls (3.31,0.3) and (6.95,1.4) .. (10.93,3.29)   ;
\draw    (70.67,89.83) -- (117.16,109.47) ;
\draw [shift={(119,110.25)}, rotate = 202.9] [color={rgb, 255:red, 0; green, 0; blue, 0 }  ][line width=0.75]    (10.93,-3.29) .. controls (6.95,-1.4) and (3.31,-0.3) .. (0,0) .. controls (3.31,0.3) and (6.95,1.4) .. (10.93,3.29)   ;
\draw    (170.17,89.83) -- (171.14,167.17) ;
\draw [shift={(171.17,169.17)}, rotate = 269.28] [color={rgb, 255:red, 0; green, 0; blue, 0 }  ][line width=0.75]    (10.93,-3.29) .. controls (6.95,-1.4) and (3.31,-0.3) .. (0,0) .. controls (3.31,0.3) and (6.95,1.4) .. (10.93,3.29)   ;
\draw   (141,170) -- (201,170) -- (201,210) -- (141,210) -- cycle ;
\draw   (90,110.5) -- (150,110.5) -- (150,150.5) -- (90,150.5) -- cycle ;
\draw   (190.5,110) -- (250.5,110) -- (250.5,150) -- (190.5,150) -- cycle ;
\draw   (241,50) -- (301,50) -- (301,90) -- (241,90) -- cycle ;
\draw   (140,50) -- (200,50) -- (200,90) -- (140,90) -- cycle ;

\draw (223.5,132) node   [align=left] {\begin{minipage}[lt]{31.28pt}\setlength\topsep{0pt}
2-UJ
\end{minipage}};
\draw (171.33,193.42) node   [align=left] {\begin{minipage}[lt]{29.47pt}\setlength\topsep{0pt}
2-UQ
\end{minipage}};
\draw (120,132.5) node   [align=left] {\begin{minipage}[lt]{28.79pt}\setlength\topsep{0pt}
2-UU
\end{minipage}};
\draw (172,69) node   [align=left] {\begin{minipage}[lt]{22.33pt}\setlength\topsep{0pt}
UQ
\end{minipage}};
\draw (273,69) node   [align=left] {\begin{minipage}[lt]{21.08pt}\setlength\topsep{0pt}
UJ
\end{minipage}};
\draw (75,70) node   [align=left] {\begin{minipage}[lt]{25.61pt}\setlength\topsep{0pt}
UU
\end{minipage}};

\end{tikzpicture}

\end{center}

\medskip
\medskip

\section{Examples and Basic Properties of 2-UQ Rings}

In the current section, we introduce the concept of 2-UQ rings and examine their elementary but useful properties. Our basic tool is the following.

\begin{definition}
We say that a ring $R$ is {\it 2-UQ} if the square of each unit is a sum of an idempotent and a quasi-nilpotent that commute with each other (equivalently, for each $u\in U(R)$, $u^2=1+q$, where $q\in QN(R)$).	
\end{definition}

The next constructions are worthwhile.

\begin{example}\label{con xemple}
(1) Every 2-UJ ring is a 2-UQ ring since $J(R) \subseteq QN(R)$. However, the converse does {\it not} necessarily hold: consider $A = \mathbb{F}_3\langle x,y\rangle$ with $x^2=0$. With the aid of \cite[Lemma 3.18]{ns}, we have:
\[ U(A) = \{\mu + \mathbb{F}_3x + xAx : \mu \in U(F_3)\}, \]
\[ Nil(A) = \{\mathbb{F}_3x + xAx\} \subseteq QN(A). \]
It also can easily be shown that $J(A) = (0)$. To see this, take $a \in J(A)$ such that $1 + ya \notin U(A)$, whence we really must have $J(A) = (0)$, as asked.

Now, we show that $A$ is a 2-UQ ring, but {\it not} 2-UJ. To that goal, let $u := \mu + \alpha x + x\beta x \in U(A)$, where $\alpha \in \mathbb{F}_3$ and $\beta \in A$. Then, one readily inspects that
\[ u^2 = \mu^2 + 2\mu(\alpha x + x\beta x) \in 1 + Nil(A) \subseteq 1 + QN(R), \]
so $A$ is, indeed, a 2-UQ ring. However, it is clear that $A$ is {\it not} 2-UJ, because $1 + x \in U(A)$, but $(1 + x)^2 = 1 - x + x^2 \notin 1 + J(A)$, as required.

(2) Every 2-UU ring is a 2-UQ ring since $Nil(R) \subseteq QN(R)$. However, the converse does {\it not} necessarily be fulfilled: consider $B = \mathbb{F}_2[[x]]$. Looking at Corollary \ref{cor five}, $B$ has to be a 2-UQ ring, but for $1 + x \in U(B)$ one checks that $(1 + x)^2 = 1 + x^2 \notin 1 + Nil(B)$. Therefore, $B$ is {\it not} a 2-UU ring, as pursued.
\end{example}

We say that $S$ is a {\it good subring} of $R$, provided $U(R) \cap S \subseteq U(S)$, which is equivalent to the inclusion map $S \to R$ being a local homomorphism. If $S$ is a good subring of $R$, then we clearly have $QN(R) \cap S \subseteq QN(S)$.

For instance, it is easy to see that $R$ is a good subring of both $R[x]$ and $R[[x]]$. However, $R[x]$ is definitely {\it not} a good subring of $R[[x]]$, because $1+x \in U(R[[x]]) \cap R[x]$ but $1+x \notin U(R[x])$.

\medskip

We continue our further work with a series of preliminary technicalities.

\begin{lemma} \label{subring}
Suppose $S$ is a good subring of a ring $R$. If $R$ is 2-UQ, then $S$ inherits this property as well.
\end{lemma}

\begin{proof}
For any unit $u \in U(S) \subseteq U(R) = 1 + QN(R)$, we have $u - 1 \in QN(R) \cap S \subseteq QN(S)$, as requested.
\end{proof}

\begin{lemma} \label{prod}
Given a family of rings $(R_i)_{i\in I}$ indexed by set $I$, the quasi-nilpotent elements of their direct product satisfy the equality
$$QN\left(\prod_{i \in I} R_i\right) = \prod_{i \in I} QN(R_i).$$
\end{lemma}

\begin{proof}
For $(a_i) \in QN(\prod R_i)$ with $a_ib_i = b_ia_i$, we have $(a_i)(b_i) = (b_i)(a_i)$, thus making $1 - (a_i)(b_i)$ invertible in $\prod R_i$. This implies that $1 - a_ib_i \in U(R_i)$, proving $a_i \in QN(R_i)$ for all $i$.

Conversely, for $(a_i) \in \prod QN(R_i)$ with $(a_i)(b_i) = (b_i)(a_i)$, we get $a_ib_i = b_ia_i$ and $1 - a_ib_i \in U(R_i)$, thus making $1 - (a_i)(b_i)$ invertible in $\prod R_i$. So, $(a_i) \in QN(\prod R_i)$, as needed.
\end{proof}

\begin{lemma}\label{product}
The direct product $\prod_{i\in I} R_i$ of rings $R_i$ is 2-UQ uniquely when each direct component ring $R_i$ is 2-UQ.
\end{lemma}

\begin{proof}
Since units in $\prod R_i$ are actually the tuples of units in $R_i$, we use Lemma \ref{prod} to obtain $QN(\prod R_i) = \prod QN(R_i)$, establishing the result.
\end{proof}

For a ring $D$ with unital subring $C$, we define the tail ring extension $\mathcal{R}[D,C]$ thus:
$$\{(d_1,...,d_n,c,c,...) | d_i \in D, c \in C, 1\leq i\leq n\},$$
which forms a ring under the traditional component-wise operations.

\begin{corollary}
$\mathcal{R}[D,C]$ is 2-UQ if and only if both $D$ and $C$ are 2-UQ rings.
\end{corollary}

\begin{lemma}\label{corner}
For any 2-UQ ring $R$ and any idempotent $e \in Id(R)$, the corner ring $eRe$ maintains the 2-UQ property.
\end{lemma}

\begin{proof}
Let $u \in U(eRe)$. Then, $u+(1-e) \in U(R)$, because
$$(u+(1-e))(u^{-1}+(1-e))=1=(u^{-1}+(1-e))(u+(1-e)).$$
Consequently, $$u^2+(1-e)=(u+(1-e))^2 \in 1+QN(R),$$ yielding that $u^2-e \in QN(R) \cap eRe$.
However, according to \cite[Lemma 3.5]{chen2012}, we deduce $$QN(R) \cap eRe \subseteq QN(eRe),$$ and thus $u^2-e \in QN(eRe)$. Therefore, $eRe$ is a 2-UQ ring, as stated.
\end{proof}

\begin{lemma}\label{matrix}
For any non-zero ring $S$ and any integer $n \geq 2$, the matrix ring $M_n(S)$ fails to be a 2-UQ ring.
\end{lemma}

\begin{proof}
We first observe that, when $n \geq 2$, $M_2(S)$ embeds as a corner subring of $M_n(S)$. Thus, it is sufficient to illustrate that $M_2(S)$ is {\it not} 2-UQ. To that end, consider the unit matrix
$$A = \begin{pmatrix} 1 & 1 \\ 1 & 0 \end{pmatrix} \in M_2(S).$$
A direct calculation enables that $A^2 - I_2 = A$, which would require $A$ to be simultaneously a unit and a quasinilpotent. However, one knows that $U(M_2(S)) \cap QN(M_2(S)) = \emptyset$, forcing the desired contradiction. Consequently, $M_n(S)$ cannot satisfy the 2-UQ property for any $n \geq 2$, as wanted.
\end{proof}

A collection $\{e_{ij} : 1 \leq i, j \leq n\}$ of non-zero elements in $R$ forms a system of $n^2$ matrix units if they satisfy the multiplication rule $e_{ij}e_{st} = \delta_{js}e_{it}$, where $\delta_{jj} = 1$ and $\delta_{js} = 0$ when $j \neq s$. In such cases, the element $e := \sum_{i=1}^n e_{ii}$ is an idempotent, and the corner ring $eRe$ is isomorphic to $M_n(S)$ for a ring $S$ defined by
$$S = \{r \in eRe : re_{ij} = e_{ij}r \text{ for all } 1 \leq i,j \leq n\}.$$

\begin{lemma} \label{dedkind finite}
All 2-UQ rings are Dedekind-finite.
\end{lemma}

\begin{proof}
Assume the contrary that $R$ is {\it not} Dedekind-finite. Then, there exist $a,b \in R$ with $ab = 1$ but $ba \neq 1$. Constructing matrix units $e_{ij} = a^i(1-ba)b^j$ and setting $e := \sum_{i=1}^n e_{ii}$, we derive $eRe \cong M_n(S)$ for some non-zero ring $S$. Since Proposition \ref{corner} guarantees that $eRe$ inherits the 2-UQ property, this would imply the same for $M_n(S)$, thus contradicting Lemma \ref{matrix}, as expected.
\end{proof}

The following statement describes division 2-UQ rings completely.

\begin{lemma}\label{division}
Let $R$ be a division ring. Then, $R$ is 2-UQ if, and only if, either $R\cong\mathbb{F}_2$ or $R\cong\mathbb{F}_3$.
\end{lemma}

\begin{proof}
Since $R$ is a division ring and $U(R) \cap QN(R) = \emptyset$, it must be that $QN(R) = (0)$. Therefore, for every $0 \neq a \in R = U(R)$, we have $a^2 = 1$, thus assuring $a^3 = a$. Utilization of the classical Jacobson's theorem (cf. \cite{lamex}) gives at once that $R$ is commutative.

Set $f(x):=1-x^2 \in R[x]$. Since $R$ is a field, the polynomial $f(x)$ has at most $2$ roots in $R^{\ast}$. So, if we suppose $A$ to be the set of all roots of $f$ in $R^{\ast}$, since \( R \) is both a 2-UQ ring and a field, it follows that, for every non-zero element \( a \in R \), \( a^2 = 1 \), thus ensuring \( R^* = A \). Therefore, $|R^{\ast}|=|A|\le 2$. So, $R$ is one of the finite fields $\mathbb{F}_2$ or $\mathbb{F}_3$. The converse being apparent, we are done.
\end{proof}

Let $R$ be a ring, and $M$ a bi-module over $R$. The trivial extension of $R$ and $M$ is putted as
\[ T(R, M) = \{(r, m) : r \in R \text{ and } m \in M\}, \]
with addition defined component-wise and multiplication defined by
\[ (r, m)(s, n) = (rs, rn + ms). \]
Note that the trivial extension $T(R, M)$ is isomorphic to the subring
\[ \left\{ \begin{pmatrix} r & m \\ 0 & r \end{pmatrix} : r \in R \text{ and } m \in M \right\} \]
of the formal $2 \times 2$ matrix ring $\begin{pmatrix} R & M \\ 0 & R \end{pmatrix}$, as well as $$T(R, R) \cong R[x]/\left\langle x^2 \right\rangle.$$ We, likewise, note that the set of units of the trivial extension $T(R, M)$ equals to \[ U(T(R, M)) = T(U(R), M). \]

The following claim characterizes quasi-nilpotent elements in various ring buildings.

\begin{lemma}\cite[Lemma 2.8]{daoa}\label{basic property}
For rings $R,S$, an $(R,S)$-bi-module $N$ and an $R$-bi-module $M$, the following containments hold:

(1) $\{(r,m) \in T(R,M) \mid r \in QN(R), m \in M\} \subseteq QN(T(R,M))$.

(2) $\left\{\begin{pmatrix} r & m \\ 0 & s \end{pmatrix} \mid r \in QN(R), s \in QN(S), m \in N\right\} \subseteq QN\left(\begin{pmatrix} R & N \\ 0 & S \end{pmatrix}\right)$.

(3) $\{(a_{ij}) \in T_n(R) \mid a_{ii} \in QN(R) \text{ for } 1 \leq i \leq n\} \subseteq QN(T_n(R))$.

(4) $\{a_0 + \cdots + a_{n-1}x^{n-1} \in R[x]/\langle x^n \rangle \mid a_0 \in QN(R)\} \subseteq QN(R[x]/\langle x^n \rangle)$.

(5) $\{a_0 + a_1x + a_2x^2 + \cdots \in R[[x]] \mid a_0 \in QN(R)\} \subseteq QN(R[[x]])$

Moreover, these inclusions become equalities when $R$ and $S$ are UQ rings.
\end{lemma}

We now immediately arrive at the following characterization of 2-UQ rings in various extensions:

\begin{corollary} \label{cor five}
Let $R$ and $S$ be rings, $N$ an $(R,S)$-bi-module and $M$ an $R$-bi-module. Then, the following equivalencies hold:

(1) The trivial extension $T(R,M)$ is 2-UQ if, and only if, $R$ is 2-UQ.

(2) The formal triangular matrix ring $\begin{pmatrix} R & N \\ 0 & S \end{pmatrix}$ is 2-UQ if, and only if, both $R$ and $S$ are 2-UQ.

(3) For any $n \geq 1$, the triangular matrix ring $T_n(R)$ is 2-UQ if, and only if, $R$ is 2-UQ.

(4) For any $n \geq 1$, the quotient ring $R[x]/\langle x^n \rangle$ is 2-UQ if, and only if, $R$ is 2-UQ.

(5) The power series ring $R[[x]]$ is 2-UQ if, and only if, $R$ is 2-UQ.
\end{corollary}

\begin{proof}
We only need to establish case (1), as the other four cases can be proved with the help of Lemma \ref{basic property} arguing as in (1). Suppose $T(R,M)$ is a 2-UQ ring. We then, with no loss of generality, can consider $R$ as a good subring of $T(R,M)$ since $R \cong T(R,0)$. Therefore, Lemma \ref{subring} tells us that $R$ is a 2-UQ ring.

Conversely, if $R$ is a 2-UQ ring and $(u,m) \in U(T(R,M))$, then $u \in U(R)$. Since $R$ is a 2-UQ ring, we have $1-u^2 \in QN(R)$. Consequently, Lemma \ref{basic property} allows us to infer that $$(1,0)-(u,m)^2 = (1-u^2, *) \in T(QN(R),M) \subseteq QN(T(R,M)).$$ So, $T(R,M)$ is a 2-UQ ring, as promised.
\end{proof}

Let $A$, $B$ be two rings, and let $M$, $N$ be an $(A,B)$-bi-module and a $(B,A)$-bi-module, respectively. Besides, we consider two bi-linear maps $\phi :M\otimes_{B}N\rightarrow A$ and $\psi:N\otimes_{A}M\rightarrow B$ that apply to the following properties:
$$Id_{M}\otimes_{B}\psi =\phi \otimes_{A}Id_{M},Id_{N}\otimes_{A}\phi =\psi \otimes_{B}Id_{N}.$$
For $m\in M$ and $n\in N$, we define $mn:=\phi (m\otimes n)$ and $nm:=\psi (n\otimes m)$. Now, the $4$-tuple $R=\begin{pmatrix}
	A & M\\
	N & B
\end{pmatrix}$ becomes to an associative ring with obvious matrix operations which is called a {\it Morita context ring}. Denote the two-sided ideals $Im \phi$ and $Im \psi$ to $MN$ and $NM$, respectively, that are called the {\it trace ideals} of the Morita context ring. Moreover, a {\it Morita context}
$\begin{pmatrix}
	A & M\\
	N & B
\end{pmatrix}$ is called {\it trivial}, provided the context products are trivial, i.e., $MN=(0)$ and $NM=(0)$. We now see that
$$\begin{pmatrix}
	A & M\\
	N & B
\end{pmatrix}\cong {\rm T}(A\times B, M\oplus N),$$
where
$\begin{pmatrix}
	A & M\\
	N & B
\end{pmatrix}$ is a trivial Morita context taking into account \cite{20}.

\medskip

We are now prepared to prove the following key instrument.

\begin{lemma}\label{morita}
Let $R= \begin{pmatrix}
A & M \\
N & B
\end{pmatrix}$ be a Morita context ring, where $MN$ and $NM$ are both nilpotent and central. Then, $R$ is a 2-$UQ$ ring if, and only if, both $A$ and $B$ are 2-$UQ$ rings.
\end{lemma}

\begin{proof}
Set $e:= \begin{pmatrix}
1 & 0 \\
0 & 0
\end{pmatrix}$. We know that $A \cong eRe$ and $B\cong (1-e)R(1-e)$. Thus, in accordance with Lemma \ref{corner}, $A$ and $B$ are both 2-$UQ$ rings.

Now, assume that $A$ and $B$ are 2-$UQ$ rings. Appealing to \cite[Lemma 3.1]{tang}, we write $U(R)= \begin{pmatrix}
U(A) & M \\
N & U(B)
\end{pmatrix}$. Choose $\begin{pmatrix}
u & m \\
n & v
\end{pmatrix} \in U(R)$. Since $A$ and $B$ are 2-$UQ$, we have $u^2=1+q$ and $v^2=1+p$, where $q \in QN(A)$ and $p \in QN(B)$.

Furthermore, because $mn$ and $nm$ are both nilpotent and central, \cite[Lemma 3.4]{daoa} informs us that $q'=q+mn \in QN(A)$ and $p'=p+nm \in QN(B)$. Thus, one extracts that
\[
\begin{pmatrix}
u & m \\
n & v
\end{pmatrix}^2=\begin{pmatrix}
u^2+mn & * \\
* & v^2+nm
\end{pmatrix}=\begin{pmatrix}
1 & 0 \\
0 & 1
\end{pmatrix}+\begin{pmatrix}
q' & * \\
* & p'
\end{pmatrix}.
\]

It, thereby, suffices to show that $Q=\begin{pmatrix}
q' & * \\
* & p'
\end{pmatrix} \in QN(R)$. To that target, let $P=\begin{pmatrix}
a & m' \\
n' & b
\end{pmatrix}$ satisfy $PQ=QP$, thus insuring $aq'-q'a \in Nil(A) \cap C(A)$.

Then, $a(aq'-q'a)=(aq'-q'a)a$, so that $a^2q'=q'a^2$. Since $(q'a^2)q'=q'(q'a^2)$ and $q' \in QN(R)$, it must be that $1-q'a^2q'\in U(A)$. Consequently,
\[
(1-q'a)(1+aq')=(aq'-q'a)+(1-q'a^2q') \in MN+U(A) \subseteq U(A),
\]
which shows $1-q'a \in U(A)$. Analogically, we can show $1-p'b \in U(B)$.

Furthermore, since $MN$ and $NM$ are nilpotent and central, we find that
\[
1-QP \in \begin{pmatrix}
U(A) & * \\
* & U(B)
\end{pmatrix}=U(R),
\]
proving that $Q \in QN(R)$, thus completing the arguments after all.
\end{proof}

Thanks to the last result, for any ring $R$, the power series ring $R[[x]]$ is 2-UQ if, and only if, $R$ is 2-UQ. However, since this is {\it not} the case with the ordinary polynomial ring, a natural question arises like this: under what extra conditions the polynomial ring $R[x]$ is a 2-UQ ring, provided that the former ring $R$ is 2-UQ? In what follows, we attempt to answer this question in some aspect. For this purpose, we first present the next useful claim.

\medskip

Let $Nil_{*}(R)$ denote the {\it prime} radical (or, in other terms, the {\it lower} nil-radical) of a ring $R$, i.e., the intersection of all prime ideals of $R$. We know that $Nil_{*}(R)$ is a nil-ideal of $R$. It is also long known that a ring $R$ is called {\it $2$-primal} if its lower nil-radical $Nil_{*}(R)$ consists of all the nilpotent elements of $R$. For example, it is well known that reduced rings and commutative rings are themselves $2$-primal.

\begin{lemma}\cite[Corollary 3.2.]{daoa}\label{cor 2 primal}
Let $R$ be a 2-primal ring. Then,
\[ J(R[x]) = QN(R[x]) = \text{Nil}(R)[x] = \text{Nil}_*(R)[x] = \text{Nil}_*(R[x]) \]
\end{lemma}

Now, we can plainly determine the necessary and sufficient condition for $R[x]$ to be a 2-UQ ring under certain additional circumstances.

\begin{lemma}
Let $R$ be a 2-primal ring. Then, the following four conditions are equivalent:
\begin{enumerate}
\item[(1)] $R$ is a 2-UU ring.
\item[(2)] $R[x]$ is a 2-UQ ring.
\item[(3)] $R[x]$ is a 2-UJ ring.
\item[(4)] $R[x]$ is a 2-UU ring.
\end{enumerate}
\end{lemma}

\begin{proof} Both equivalencies (3) $\Rightarrow$ (2) and (4) $\Rightarrow$ (2) are pretty obvious, since we always have $Nil(R)\subseteq QN(R)$ and $J(R) \subseteq QN(R)$. The equivalence (2) $\Leftrightarrow$ (3) follows directly from Corollary \ref{cor 2 primal}.

(1) $\Rightarrow$ (4): Write $u := \sum_{i=0}^n u_ix^i \in U(R[x])$. Since $R$ is 2-primal, in view of \cite[Theorem 2.5]{Chenpr}, we have $u_0 \in U(R)$ and, for each $1 \leq i \leq n$, $u_i \in \text{Nil}_*(R)$. So, since $R$ is 2-UU, we discover $1-u_0^2 \in \text{Nil}(R)$ and, by the ideal property of $\text{Nil}_*(R)$, we obtain:
\[ 1-u^2 \in (1-u_0^2) + \text{Nil}_*(R)[x]x \subseteq \text{Nil}_*(R)[x] = \text{Nil}_*(R[x]) \subseteq \text{Nil}(R[x]).\]

(4) $\Rightarrow$ (1): Let $u \in U(R) \subseteq U(R[x])$. Then, $1-u^2 \in \text{Nil}(R[x]) \cap R \subseteq \text{Nil}(R)$. Hence, $R$ is also a 2-UU ring.

(2) $\Rightarrow$ (1): Let $u \in U(R) \subseteq U(R[x])$. Then, $1-u^2 \in QN(R[x]) = \text{Nil}_*(R)[x]$. Therefore, $1-u^2 \in \text{Nil}_*(R) = \text{Nil}(R)$, as requested.
\end{proof}

We now need to remember the following valuable assertion.

\begin{lemma}\cite[Lemma 4.1]{daoa} \label{QNR}
Let $R$ be a ring, $I$ an ideal of $R$ such that $I \subseteq J(R)$, and $\overline{R} = R/I$. Then, the following two points are fulfilled:

(1) For every $\bar{q} \in QN(\overline{R})$, we have $q \in QN(R)$.

(2) For every $\bar{q} \in QN(\overline{R})$ and $p \in I$, we have $q + p \in QN(R)$.
\end{lemma}

As a consequence, we detect:

\begin{corollary}\label{factor UQ}
Let $R$ be a ring with an ideal $I \subseteq J(R)$. If the quotient ring $\overline{R} = R/I$ is 2-UQ, then $R$ itself is a 2-UQ ring.
\end{corollary}

\begin{proof}
Take an arbitrary unit $u \in U(R)$. Thus, its image $\bar{u}$ remains a unit in $\overline{R}$. Since $\overline{R}$ is 2-UQ, we have $\bar{u}^2 = \bar{1} + \bar{q}$ for some $\bar{q} \in QN(\overline{R})$. In virtue of Lemma \ref{QNR}(1), the pre-image $q$ of the lifted element $\bar{q}$ belongs to $QN(R)$.

But, the difference $u^2 - (1 + q)$ lies in $I$, so we can write $u^2 = 1 + (q + p)$ for some $p \in I \subseteq J(R)$. Now, Lemma \ref{QNR}(2) is a guarantor that $q + p \in QN(R)$, proving that $R$ is, indeed, a 2-UQ ring, as formulated.
\end{proof}

Our next criterion sounds thus:

\begin{lemma}\label{local and semisimple}
(1) Let $R$ be a 2-UQ ring. Then, $R$ is local if, and only if, either $R/J(R) \cong \mathbb{F}_2$ or $R/J(R) \cong \mathbb{F}_3$.

(2) A semi-simple ring $R$ is 2-UQ if, and only if, $$R \cong \mathbb{F}_p \times \cdots \times \mathbb{F}_p \times \mathbb{F}_q \times \cdots \times \mathbb{F}_q,$$ where $p,q \in \{2,3\}$.

(3) Let $R$ be a 2-UQ ring. Then, $R$ is semi-local if, and only if, $$R/J(R) \cong \mathbb{F}_p \times \cdots \times \mathbb{F}_p \times \mathbb{F}_q \times \cdots \times \mathbb{F}_q,$$ where $p,q \in \{2,3\}$.
\end{lemma}

\begin{proof}
(1) If either $R/J(R) \cong \mathbb{F}_2$ or $R/J(R) \cong \mathbb{F}_3$, it follows at once that $R$ is a local ring.
Conversely, assuming $R$ is a local ring, one can elementarily show that $QN(R)=J(R)$, and thus via \cite[Example 2.7]{cui} the proof is complete.

(2) If $R$ is a semi-simple ring, exploiting the celebrated Wedderburn-Artin theorem (see, for instance, \cite{lamex}), we may write $R\cong \prod M_{n_i}(D_i)$. Consequently, combining Lemmas \ref{product}, \ref{matrix} and \ref{matrix}, we conclude that $$R \cong \mathbb{F}_p \times \cdots \times \mathbb{F}_p \times \mathbb{F}_q \times \cdots \times \mathbb{F}_q,$$ as asked for.

(3) It is clear by a combination of (1) and (2).
\end{proof}

We finish this section with the following.

\begin{example}
The ring $\mathbb{Z}_m$ is 2-UQ if, and only if, $m = 2^k3^s$ for some positive integers $k$ and $s$.
\end{example}

\begin{proof}
Write $m = p_1^{k_1} \cdots p_n^{k_n}$, where $p_1, \ldots, p_n$ are distinct primes, and $k_1, \ldots, k_n$ are positive integers. So, one has that $\mathbb{Z}_m \cong \mathbb{Z}_{p_1^{k_1}} \times \cdots \times \mathbb{Z}_{p_n^{k_n}}$. The result then follows by Lemmas \ref{local and semisimple}(1) and \ref{division}, as claimed.
\end{proof}

\section{Semi-Potent 2-UQ Rings}

This section mainly focuses on the exploration of connections of 2-UQ rings with regular rings, semi-potent rings, potent rings, etc. Let us recollect for completeness of the exposition that a ring $R$ is called {\it semi-potent} if every one-sided ideal {\it not} contained in $J(R)$ contains a non-zero idempotent. Moreover, a semi-potent ring $R$ is called {\it potent}, provided that all idempotents lift modulo $J(R)$.

\medskip

Our first two technical claims here are the following ones.

\begin{lemma} \label{sum two unit}
Let $R$ be a 2-UQ ring, and let $\overline{R} = R/J(R)$. The following two items hold:

(1) For any $u, v \in U(R)$, we have $u^2 + v \neq 1$.

(2) For any $\bar{u}^2, \bar{v} \in U(\overline{R})$, we have $\bar{u}^2 + \bar{v} \neq \bar{1}$.
\end{lemma}

\begin{proof}
(1) Suppose $u^2 + v = 1$. Since $R$ is a 2-UQ ring, there exists $q \in QN(R)$ such that $u^2 = 1 + q$. Thus, $1 = u^2 + v = 1 + q + v$, giving $q = -v \in QN(R) \cap U(R)$, a contradiction.

(2) Suppose $\overline{u}^2 + \overline{v} = \overline{1}$. We may assume $u, v \in U(R)$, whence $u^2 + v - 1 \in J(R)$. On the other hand, since $R$ is a 2-UQ ring, there exists $q \in QN(R)$ such that $u^2 = 1 + q$, implying $q + v \in J(R)$. Therefore, $q \in U(R) + J(R) \subseteq U(R)$, again a contradiction.
\end{proof}

\begin{lemma}\label{exe}
Let $R$ be a potent 2-$UQ$ ring, and $\overline{R} = R/J(R)$. Then, the following two issues hold:

(1) For any idempotent $\bar{e} \in \overline{R}$ and units $\bar{u}, \bar{v} \in U(\bar{e}\overline{R}\bar{e})$, it must be that $\bar{u}^2 + \bar{v} \neq \bar{e}$.

(2) There are no idempotents $\bar{e} \in \overline{R}$ for which $\bar{e}\overline{R}\bar{e}$ is isomorphic to $M_2(S)$ for any ring $S$.
\end{lemma}

\begin{proof}
(1) For an arbitrary idempotent $\bar{e} \in \overline{R}$, we can lift it to an idempotent $e \in R$ since by assumption all idempotents lift modulo the Jacobson radical. But, the corner ring satisfies the isomorphism $$\bar{e}\overline{R}\bar{e} \cong eRe/J(eRe).$$ However, as $eRe$ inherits the 2-$UQ$ property from $R$ with Lemma \ref{corner}) at hand, the result follows immediately from Lemma~\ref{sum two unit}(2).

(2) Consider the matrix identity in $M_2(S)$
    \[
    I_2 = \begin{pmatrix} 1 & 1 \\ 1 & 0 \end{pmatrix}^2 + \begin{pmatrix} -1 & -1 \\ -1 & 0 \end{pmatrix},
    \]
which expresses the identity matrix as a sum of units. Thus, this would provide a solution to the equation $\bar{u}^2 + \bar{v} = \bar{e}$ in $\bar{e}\overline{R}\bar{e} \cong M_2(S)$, contradicting part (1).
\end{proof}

In the following statement, we examine the relationships between 2-UQ, 2-UJ and 2-UU rings, respectively, whenever $R$ is a semi-potent ring.

\begin{theorem}\label{semipotent}
Let $R$ be a semi-potent ring. The following are equivalent:

(1) The factor-ring $R/J(R)$ is a 2-UQ ring.

(2) The factor-ring $R/J(R)$ is a tripotent ring.

(3) The factor-ring $R/J(R)$ is a 2-UU ring.

(4) The ring $R$ is a 2-UJ ring.

\end{theorem}

\begin{proof}
Implications (4) $\Rightarrow$ (3) $\Rightarrow$ (1) are quite obvious, so we omit the details.

(1) $\Rightarrow$ (2): We first show that the quotient $R/J(R)$ is reduced. Suppose on the contrary that $x^2=0$ with $0 \neq x \in R/J(R)$. Since $R$ is semi-potent, $R/J(R)$ is also semi-potent. Thus, viewing \cite[Theorem 2.1]{lev}, there exists $e \in R/J(R)$ such that $eR/J(R)e \cong M_2(S)$, where $S$ is a non-zero ring.

As $R/J(R)$ is a 2-UQ ring, Lemma \ref{corner} works to get that $eR/J(R)e$ is too a 2-UQ ring and, consequently, $M_2(S)$ is a 2-UQ ring as well. However, this contradicts Lemma \ref{matrix}. Therefore, $R/J(R)$ is reduced, as desired.

Now, suppose there exists $a \in R/J(R)$ such that $a-a^3 \neq 0$ in $R/J(R)$. Since $R/J(R)$ is semi-potent, there exists $e=e^2 \in R/J(R)$ with $e \in (a-a^3)R/J(R)$. We write $e=(a-a^3)b$ for some $b \in R/J(R)$.

But, as $R/J(R)$ is abelian, we also deduce that
\begin{equation*}
e = ea(1-a^2)b = e(1-a^2)ab.
\end{equation*}
Thus, Lemma \ref{dedkind finite} applies to infer that $ea, e(1-a^2) \in U(eR/J(R)e)$.

On the other hand, one verifies that
\begin{equation*}
(ea)^2 + e(1-a^2) = e. \qquad (*)
\end{equation*}
Since $R/J(R)$ is a 2-UQ ring, Lemma \ref{corner} is again applicable to find that $eR/J(R)e$ also is 2-UQ. Finally, equation $\left( \ast \right)$ contradicts Lemma \ref{sum two unit}, and hence $R/J(R)$ is tripotent.

(2) $\Rightarrow$ (4): Let $u \in U(R)$. Then, by (2), we have $u-u^3 \in J(R)$. Moreover, since $J(R)$ is an ideal and $u$ is a unit, we receive $1-u^2 \in J(R)$. Therefore, $R$ is a 2-UJ ring, as asserted.
\end{proof}

A ring $R$ is known to be {\it $\pi$-regular} when, for every element $a \in R$, there exists $n \geq 1$ such that $a^n$ lies in the principal ideal $a^nRa^n$, thus generalizing the notion of a regular ring in which this relation holds for $n=1$. Strengthening this concept, a ring $R$ is called {\it strongly $\pi$-regular} if, for each $a \in R$, we have $a^n \in a^{n+1}R$ for some $n \geq 1$. Likewise, $R$ is {\it reduced} if it contains no non-zero nilpotent elements, that is, $Nil(R) = (0)$.

\medskip

We thus extract the following consequence.

\begin{corollary}
The following five statements are equivalent for any 2-UQ ring \(R\):

(1) $R$ is a regular ring.

(2) $R$ is a $\pi$-regular reduced ring.

(3) $R$ is a strongly regular ring.

(4) $R$ is a unit-regular ring.

(5) $R$ is a tripotent ring.
\end{corollary}

\begin{proof}
(1) $\Rightarrow$ (2): Since \( R \) is a regular ring, \( J(R) = (0) \). Moreover, \( R \) is a semi-potent ring, because, for every \( 0 \neq a \in R \), there exists \( b \in R \) such that \( a = aba \). Obviously, \( e := ab \in aR \) is an idempotent. We now show that \( R \) is reduced. Assume, on the contrary, that there exists a non-zero element \( a \in R \) such that \( a^2 = 0 \). Applying \cite[Theorem 2.1]{lev}, there exists an idempotent \( e \in RaR \) such that \( eRe \cong M_2(S) \) for some non-trivial ring \( S \). Since \( R \) is a 2-UQ ring, Lemma \ref{corner} employs to confirm that \( eRe \) is a 2-UQ ring, too. Thus, \( M_2(S) \) has to be a 2-UQ ring, as well, which contradicts Lemma \ref{matrix}. Therefore, \( a = 0 \), showing that \( R \) is a reduced ring.

(2) $\Rightarrow$ (1): This follows automatically from \cite[Theorem 3]{badw}.

(1) $\Rightarrow$ (3): Taking in mind (1) $\Rightarrow$ (2), we may assume that $R$ is reduced (and hence abelian). Therefore, $R$ is a strongly regular ring.

(3) $\Rightarrow$ (4): This is always true bearing into account \cite[Exercise 12.6C]{lamex}).

(4), (5) $\Rightarrow$ (1): These are self-evident.

(1) $\Rightarrow$ (5): Since every regular ring is, itself, semi-potent and $J(R)=(0)$, the wanted result follows at once from Theorem \ref{semipotent}.
\end{proof}

Now, we have at our disposal everything to prove our next principal result in this section.

\begin{theorem}\label{exe 2}
Let $R$ be a potent ring. Then, the following six conditions are equivalent:

(1) $R$ is a 2-$UQ$ ring.

(2) $R/J(R)$ is a 2-$UQ$ ring.

(3) $R/J(R)$ is a tripotent ring.

(4) $R$ is a 2-$UJ$ ring.

(5) $R/J(R)$ is a 2-$UJ$ ring.

(6) $R/J(R)$ is a 2-$UU$ ring.
\end{theorem}

\begin{proof}
One sees that Theorem \ref{semipotent} establishes the equivalencies (2) $\Leftrightarrow$ (3) $\Leftrightarrow$ (4) $\Leftrightarrow$ (6), whereas \cite[Lemma 2.5.]{cui} proves the equivalence (4) $\Leftrightarrow$ (5). Finally, the implication (4) $\Rightarrow$ (1) is immediate, thus leaving only (1) $\Rightarrow$ (3) to be verified in what follows.

(1) $\Rightarrow$ (3): Since $R$ is a semi-potent ring, $\overline{R} = R/J(R)$ is also semi-potent. We will show that $\overline{R}$ is a reduced ring. Suppose the opposite, namely that $x^2 = 0$ but $0 \neq x \in \overline{R}$. Then, consulting with \cite[Theorem 2.1]{lev}, there exists $\bar{e} \in \overline{R}$ such that $\bar{e}\overline{R}\bar{e} \cong M_2(S)$, where $S$ is a non-zero ring. This, however, contradicts Lemma \ref{exe}(2). Hence, $\overline{R}$ is reduced, as claimed.

Now, suppose there exists $a \in \overline{R}$ with $a - a^3 \neq 0$ in $\overline{R}$. Since $\overline{R}$ is semi-potent, there exists $0\neq e = e^2 \in \overline{R}$ such that $e \in (a - a^3)\overline{R}$. Thus, $e = (a - a^3)b$ for some $b \in \overline{R}$. Since $e$ is central, we have $e = ea \cdot e(1-a^2) \cdot ey$, so both $ea, e(1-a^2) \in U(e\overline{R}e)$. Consequently, $(ea)^2 + e(1-a^2) = e$, contradicting Lemma \ref{exe}(1). Therefore, $\overline{R}$ is tripotent, as asserted.
\end{proof}

Referring to  Ko\c{s}an {\it et al.} \cite{kosanz}, an element of a ring is called {\it semi-tripotent} if it can be expressed as the sum of a tripotent element and an element from its Jacobson radical. So, a ring is called {\it semi-tripotent} whenever every element is semi-tripotent, which is equivalent to the conditions that $R/J(R)$ is tripotent and idempotents lift modulo $J(R)$.

\medskip

In what follows, we present several simple but substantial consequences of Theorem \ref{exe 2}, starting with the following.

\begin{corollary}
Let $R$ be a ring. Then, $R$ is a 2-UQ clean ring if, and only if, $R$ is semi-tripotent.
\end{corollary}

\begin{proof}
Regarding Theorem \ref{exe 2}, it follows that every clean 2-UQ ring is semi-tripotent.

Conversely, assume $R$ is semi-tripotent. Choose $u \in U(R)$ and let $u = e + j$ be a semi-tripotent representation of $u$. Then, one finds that
\[
e = u - j \in U(R) + J(R) \subseteq U(R) \Rightarrow e^2 \in U(R) \cap \text{Id}(R) = \{1\}
\]
and, consequently,
\[
u^2 = e^2 + (ej + je + j^2) \in 1 + J(R) \subseteq 1 + QN(R).
\]
We now peove that $R$ is a clean ring. Indeed, since $R/J(R)$ is tripotent, it follows that $R/J(R)$ is strongly regular and, consequently, employing \cite[Theorem 1.1.2]{chenbook}, it is clean. Moreover, with the help of the definition of semi-tripotent rings, idempotents lift modulo $J(R)$. Therefore, $R$ is a clean ring, as stated.
\end{proof}

A significant consequence is the following.

\begin{corollary}
Let $R$ be an artinian (in particular, a finite) ring. Then, the following three conditions are equivalent:

(1) $R$ is a 2-$UQ$ ring.

(2) $R$ is a 2-$UJ$ ring.

(3) $R$ is a 2-$UU$ ring.
\end{corollary}

\begin{proof}
Our argument treats, specifically, artinian rings, because finite rings constitute a proper subclass of artinian rings. So, from \cite[Corollary 6]{camilo}, it follows that all artinian rings possess the clean property. This ensures, via Corollary \ref{exe 2}, that $R$ is 2-$UQ$ if, and only if, it is 2-$UJ$.

The artinian hypothesis further guarantees the containment $J(R) \subseteq N(R)$. Likewise, reference to \cite[Lemma 2.1]{10} establishes that $R$ meets the 2-$UU$ condition precisely when its quotient $R/J(R)$ does. Combining these observations through Corollary \ref{exe 2} yields the equivalence between $R$ being 2-$UJ$ and 2-$UU$, finishing the proof.
\end{proof}

We over this section with the following exhibitions.

\begin{example}
Here, we manage to present certain variations for the combination of 2-UQ rings with rings such as potent, semi-potent, regular and artinian, respectively.

(1) We know that the property of being a 2-UQ ring and a potent ring is tantamount to the property of a ring $R$ for which \( R/J(R) \) is tripotent. Therefore, these rings are equivalent to being semi-tripotent rings. Hence, rings such as \( \mathbb{Z}_2[[x]] \), \( \mathbb{Z}_3[[x]] \), \( T_n(\mathbb{Z}_2) \) and \( T_n(\mathbb{Z}_3) \), where \( n \in \mathbb{N} \), all have this property.\\
(2) Also, for combining 2-UQ and semi-potent rings, one can write an argument similar to (1).\\
(3) We know for a ring that being simultaneously regular and 2-UQ is equivalent to the tripotent property. Therefore, these rings are sub-direct products of rings of types \( \mathbb{Z}_2 \)'s and \( \mathbb{Z}_3 \)'s. Hence, rings such as \( \prod_{i \in I} \mathbb{Z}_2 \), \( \prod_{i \in I} \mathbb{Z}_3 \) and \( \prod_{i \in I} \mathbb{Z}_2 \times \prod_{j \in J} \mathbb{Z}_3 \), where the size can be both finite and infinite, all have this property.\\
(4) Artinian rings that are 2-UQ are equivalent to being 2-UU rings. Therefore, rings such as \( T_n(\mathbb{Z}_2) \) and \( T_n(\mathbb{Z}_3) \), where \( n \in \mathbb{N} \), all have this property.
\end{example}

\section{2-UQ Group Rings}

In this section, we explore the behavior of the 2-UQ property within the context of group rings. By examining specific conditions on the underlying groups and utilizing the structural properties of augmentation ideals, we offer deeper insights into how the 2-UQ framework interacts with group-theoretic constructions of the former group.

Let $R$ be any ring and $G$ any group. We denote by $RG$ the {\it group ring} of $G$ over $R$. The { \it augmentation ideal} $\Delta(RG)$ is standardly defined as the kernel of the augmentation homomorphism $\varepsilon: RG\to R$, where $\varepsilon(\sum_{g\in G}a_g g) = \sum_{g\in G}a_g$.

Imitating the traditional terminology, we say that a group $G$ is a {\it $p$-group} if the order of every element of $G$ is a power of the prime number $p$. Moreover, a group $G$ is said to be {\it locally finite} if each finitely generated subgroup is finite.

\medskip

We first establish two foundational results:

\begin{lemma}\label{4.14} \cite[Lemma 2]{26}.
For a prime $p\in J(R)$ and a locally finite $p$-group $G$, the augmentation ideal satisfies the inclusion $\Delta(RG) \subseteq J(RG)$.
\end{lemma}

\begin{proposition}
(1) If $RG$ is a 2-$UQ$ ring, then $R$ is also a 2-$UQ$ ring.

(2) If $R$ is a 2-$UQ$ ring and $G$ is a locally finite $p$-group, where $p$ is a prime number such that $p \in J(R)$, then $RG$ is a 2-$UQ$ ring.
\end{proposition}

\begin{proof}
(1) Since $R$ embeds naturally in $RG$, Lemma \ref{subring} forces that $R$ inherits the 2-$UQ$ property via $RG$.

(2) Observe that Lemma \ref{4.14} establishes the containment $\Delta(RG)\subseteq J(RG)$. Now, using the isomorphism $RG/\Delta(RG) \cong R$ combined with Corollary \ref{factor UQ}, we deduce the conclusion.
\end{proof}

An interesting converse relationship is the following one.

\begin{proposition} \label{group ring 2 in Jacobson}
For a group ring $RG$ with 2-$UQ$ property and $2\in J(R)$, the group $G$ must be a 2-group.
\end{proposition}

\begin{proof}
First, we show that $2 \in J(RG)$. Since $2 \in J(R)$, we have $3 \in U(R) \subseteq U(RG)$. Therefore, $9=3^2=1+QN(RG)$, which means $2^3=8 \in QN(RG)$. According to \cite{badw}, we have $2 \in QN(RG)$ and, since $2 \in C(RG)$, it follows that $2 \in J(RG)$.

We now intend to prove two claims:

\medskip

\textbf{Claim 1:} All group elements have finite order.

\medskip

Since one inspects that $R\langle g\rangle$ is a good subring of $RG$, it follows that $R\langle g\rangle$ is also 2-$UQ$. Without loss of generality, assume the element $g$ to be central. Now, suppose on the reciprocity, $g$ has infinite order. Since $RG$ is a 2-$UQ$ ring, we have $1-g^2 \in QN(RG)$. Given $2 \in J(RG)$, thanking to \cite[Proposition 2.7]{cuit}, we can write $1+g^2 \in QN(RG)$, thus assuring that $1+g+g^2 \in U(RG)$. Therefore, there exist integers $n < m$ and some elements $a_i$ with $a_n \neq 0 \neq a_m$ such that
\[
(1+g+g^2)\sum_{i=n}^{m} a_ig^i=1.
\]
This, however, leads to a contradiction, and thus every element $g \in G$ must have finite order.

\medskip

\textbf{Claim 2:} For any \( g \in G \) and \( k \in \mathbb{N} \), we have \(\sum_{i=0}^{2k} g^i \in U(RG)\).

\medskip

We will demonstrate this statement only for \( k=1, 2 \), as the general case follows by analogy.

For \( k=1 \) and any \( g \in G \), we have \( 1-g^2 \in QN(RG) \), since it can easily be shown that, if \( q \in QN(R) \cap C(R) \), then \( q \in J(R) \). Given that \( 2 \in J(RG) \), we can write \( 1+g^2 \in QN(RG) \) and, consequently, \( 1+g+g^2 \in U(RG) \).

For \( k=2 \) and any \( g \in G \), observing that \( g, g^2 \in U(RG) \), we obtain \( 1-g^2 \in QN(RG) \) and thus \( 1+g^2 \in QN(RG) \). Therefore, \( g+g^3 \in QN(RG) \). Moreover, \( 1-g^4 \in QN(RG) \) and hence \( 1+g^4 \in QN(RG) \).

Furthermore, we obtain \[ 2+g+g^2+g^3+g^4 \in QN(RG) \] and, because \( 2 \in J(RG) \), this means that
\[ g+g^2+g^3+g^4 \in QN(RG) \] i.e., consequently, \[ 1+g+g^2+g^3+g^4 \in U(RG). \]

Continuing this process, we can show that \(\sum_{i=0}^{2k} g^i \in U(RG)\), as claimed.

Now, if \( g \in G \) has order \( p \), where \( p \) does not divide 2, then \( p \) must be odd, that is, \( p-1=2k \). Consequently, by what we have shown above, \(\sum_{i=0}^{2k} g^i \in U(RG)\) and, since \( (1-g)(\sum_{i=0}^{2k} g^i)=0 \), we derive that \( 1-g=0 \), which is a contradiction. Thus, \( G \) must be a 2-group, as stated.
\end{proof}

We are now able to proceed by proving the following major assertion.

\begin{theorem}\label{23}
Let $R$ be a ring with $3 \in J(R)$, and $G$ a $p$-group with $p$ a prime. If $RG$ is 2-UQ, then either $G$ is a 3-group or $G$ is a group of exponent 2.
\end{theorem}

\begin{proof}
Following a similar approach to that used in Proposition \ref{group ring 2 in Jacobson}, we can show that $3 \in J(RG)$. Now, suppose $g \in G$. Since, as noticed above, $R\langle g\rangle$ is a good subring of $RG$, we will work with $R\langle g\rangle$ instead of $RG$ in the continuation of the evidence. Therefore, with no harm in generality, we may assume that $G$ is an abelian group. We next consider two possible cases, namely when $p=2$ and when $p \neq 2$.

\medskip

\textbf{Case 1:} If $p=2$, then $G$ is a 2-group. For any $g \in G$, there exists $n \in \mathbb{N}$ such that $g^{2^n}=1$. Let $n>1$ be the smallest such integer. Since $2^{n-1}$ is even, there exists $m \in \mathbb{N}$ with $2^{n-1}=2m$, and so
\[
1-g^{2^{n-1}}=1-g^{2m} \in QN(RG).
\]
Then, since $3 \in J(RG)$, an appeal to \cite[Lemma 3.4(2)]{daoa} gives that
\[
1+2g^{2^{n-1}}=1-g^{2^{n-1}}+3g^{2^{n-1}} \in QN(RG)
\]
and, consequently, due to the centrality of $g$, we obtain:
\[
1+g^{2^{n-1}}=1+2g^{2^{n-1}}-g^{2^{n-1}} \in QN(RG) + U(RG) \subseteq U(RG).
\]
However, $$(1-g^{2^{n-1}})(1+g^{2^{n-1}})=1-g^{2^{n}}=1-1=0,$$ and, since $1+g^{2^{n-1}}\in U(RG)$, we must have $g^{2^{n-1}}=1$, thus contradicting the minimality of $n$. Finally, $n=1$.

\medskip

\textbf{Case 2:} If $p \neq 2$, then $p$ must be odd. Let $g \in G$ with $g^{p^n}=1$. Suppose $p^n=2k+1$. Then, $1-g^{2k} \in QN(RG)$, and since $g$ is central, by \cite[Lemma 3.4(2)]{daoa}, we have
\[
1-g=-g(1-g^{2k})\in QN(RG).
\]
On the other hand, since $1-g$ is central, we get $1-g \in J(RG)$. This teaches that $\Delta(RG) \subseteq J(RG)$ (remember that we replaced the whole group ring $RG$ by the ring $R\langle g\rangle$, as well as that $1-g^n=(1-g)(1+g + \cdots + g^{n-1}) \in J(RG)$). Activating \cite[Proposition 15(i)]{coon}, one infers that $G$ is a $q$-group with $q \in J(R)$. Since $3 \in J(R)$ and knowing that two distinct primes cannot be in $J(R)$, we must have $q=3$. Thus, there exists $m \in \mathbb{N}$ such that $g^{3^m}=1$. As $g \in G$ was chosen arbitrary, $G$ is really a 3-group, as we pursued.
\end{proof}

\end{document}